\documentclass[notitlepage,11pt,reqno]{amsart}
\usepackage[foot]{amsaddr}
\usepackage[numbers,sort]{natbib}

\usepackage{amssymb,nicefrac,bm,upgreek,mathtools,verbatim}
\usepackage[final]{hyperref}
\usepackage[mathscr]{eucal}
\usepackage{dsfont}

\usepackage{color}
\definecolor{dmagenta}{rgb}{.4,.1,.5}
\definecolor{dblue}{rgb}{.0,.0,.5}
\definecolor{mblue}{rgb}{.0,.0,.8}
\definecolor{ddblue}{rgb}{.0,.0,.4}
\definecolor{dred}{rgb}{.6,.0,.0}
\definecolor{dgreen}{rgb}{.0,.5,.0}
\definecolor{Eeom}{rgb}{.0,.0,.5}

\usepackage[normalem]{ulem}

\usepackage[margin=1in]{geometry}
\allowdisplaybreaks

\newtheorem{lemma}{Lemma}[section]
\newtheorem{theorem}{Theorem}[section]
\newtheorem{proposition}{Proposition}[section]

\theoremstyle{definition}

\newtheorem{remark}{Remark}[section]

\numberwithin{equation}{section}

\hypersetup{
  colorlinks=true,
  citecolor=dblue,
  linkcolor=dblue,
  frenchlinks=false,
  pdfborder={0 0 0},
  naturalnames=false,
  hypertexnames=false,
  breaklinks}
\usepackage[capitalize,nameinlink]{cleveref}

\crefname{section}{Section}{Sections}
\crefname{subsection}{Subsection}{Subsections}
\crefname{condition}{Condition}{Conditions}
\crefname{hypothesis}{Hypothesis}{Conditions}
\crefname{assumption}{Assumption}{Assumptions}
\crefname{lemma}{Lemma}{Lemmas}
\crefname{claim}{Claim}{Claims}

\Crefname{figure}{Figure}{Figures}

\crefformat{equation}{\textup{#2(#1)#3}}
\crefrangeformat{equation}{\textup{#3(#1)#4--#5(#2)#6}}
\crefmultiformat{equation}{\textup{#2(#1)#3}}{ and \textup{#2(#1)#3}}
{, \textup{#2(#1)#3}}{, and \textup{#2(#1)#3}}
\crefrangemultiformat{equation}{\textup{#3(#1)#4--#5(#2)#6}}%
{ and \textup{#3(#1)#4--#5(#2)#6}}{, \textup{#3(#1)#4--#5(#2)#6}}%
{, and \textup{#3(#1)#4--#5(#2)#6}}

\Crefformat{equation}{#2Equation~\textup{(#1)}#3}
\Crefrangeformat{equation}{Equations~\textup{#3(#1)#4--#5(#2)#6}}
\Crefmultiformat{equation}{Equations~\textup{#2(#1)#3}}{ and \textup{#2(#1)#3}}
{, \textup{#2(#1)#3}}{, and \textup{#2(#1)#3}}
\Crefrangemultiformat{equation}{Equations~\textup{#3(#1)#4--#5(#2)#6}}%
{ and \textup{#3(#1)#4--#5(#2)#6}}{, \textup{#3(#1)#4--#5(#2)#6}}%
{, and \textup{#3(#1)#4--#5(#2)#6}}

\crefdefaultlabelformat{#2\textup{#1}#3}

\makeatletter
\newcommand{\refcheckize}[1]{%
  \expandafter\let\csname @@\string#1\endcsname#1%
  \expandafter\DeclareRobustCommand\csname relax\string#1\endcsname[1]{%
    \csname @@\string#1\endcsname{##1}\@for\@temp:=##1\do{\wrtusdrf{\@temp}%
\wrtusdrf{{\@temp}}}}%
  \expandafter\let\expandafter#1\csname relax\string#1\endcsname}
\newcommand{\refcheckizetwo}[1]{%
  \expandafter\let\csname @@\string#1\endcsname#1%
  \expandafter\DeclareRobustCommand\csname relax\string#1\endcsname[2]{%
    \csname @@\string#1\endcsname{##1}{##2}\wrtusdrf{##1}\wrtusdrf{{##1}}%
\wrtusdrf{##2}\wrtusdrf{{##2}}}%
  \expandafter\let\expandafter#1\csname relax\string#1\endcsname}
\makeatother

\refcheckize{\cref}
\refcheckize{\Cref}
\refcheckizetwo{\crefrange}
\refcheckizetwo{\Crefrange}

\makeatletter
\DeclareRobustCommand\widecheck[1]{{\mathpalette\@widecheck{#1}}}
\def\@widecheck#1#2{%
    \setbox\z@\hbox{\m@th$#1#2$}%
    \setbox\tw@\hbox{\m@th$#1%
       \widehat{%
          \vrule\@width\z@\@height\ht\z@
          \vrule\@height\z@\@width\wd\z@}$}%
    \dp\tw@-\ht\z@
    \@tempdima\ht\z@ \advance\@tempdima2\ht\tw@ \divide\@tempdima\thr@@
    \setbox\tw@\hbox{%
       \raise\@tempdima\hbox{\scalebox{1}[-1]{\lower\@tempdima\box
\tw@}}}%
    {\ooalign{\box\tw@ \cr \box\z@}}}

\def\subsection{\@startsection{subsection}{0}%
\z@{\linespacing\@plus\linespacing}{\linespacing}%
{\bf}}

\makeatother

\DeclareMathOperator{\Exp}{\mathbb{E}} 
\DeclareMathOperator{\Prob}{\mathbb{P}} 
\newcommand{\D}{\mathrm{d}}          
\newcommand{\RR}{\mathbb{R}}         
\newcommand{\Rd}{{\mathbb{R}^d}}       
\newcommand{\NN}{\mathbb{N}}         
\newcommand{\Ind}{\mathds{1}}            

\newcommand{\grad}{\nabla}

\newcommand{\cA}{\mathcal{A}}
\newcommand{\sB}{\mathscr{B}}    
\newcommand{\cB}{\mathrm{B}}
\newcommand{\cK}{\mathrm{K}}

\newcommand{\cC}{\mathcal{C}}     
\newcommand{\cD}{\mathrm{D}}     
\newcommand{\cU}{\mathrm{U}}
\newcommand{\cN}{\mathrm{N}}
\newcommand{\cF}{\mathcal{F}}
\newcommand{\sF}{\mathfrak{F}}    

\newcommand{\sM}{\mathscr{M}}     
\newcommand{\cO}{\mathcal{O}}

\newcommand{\cS}{\mathcal{S}}
\newcommand{\cT}{\mathcal{T}}

\newcommand{\abs}[1]{\lvert#1\rvert}
\newcommand{\norm}[1]{\lVert#1\rVert}

\newcommand{\sgn}{\text{sgn}}

\providecommand{\pro}[1]{(#1_t)_{t \geq 0}}

\providecommand{\seq}[1]{(#1_n)_{n\in \mathbb{N}}}

\newcommand{\fL}{(-\Delta)^{s}}

\DeclareMathOperator{\dist}{dist}

\DeclareMathOperator{\Deg}{\mathrm{deg}}

\begin{document}

\title[Fractional Ambrosetti-Prodi results]%
{\sc \textbf{Ambrosetti-Prodi type results for Dirichlet problems of fractional Laplacian-like operators}}

\author{Anup Biswas and J\'{o}zsef L\H{o}rinczi}

\address{Anup Biswas \\
Department of Mathematics, Indian Institute of Science Education and Research, Dr. Homi Bhabha Road,
Pune 411008, India, anup@iiserpune.ac.in}

\address{J\'ozsef L\H{o}rinczi \\
Department of Mathematical Sciences, Loughborough University, Loughborough LE11 3TU, United Kingdom,
J.Lorinczi@lboro.ac.uk}

\date{}


\begin{abstract}
We establish Ambrosetti--Prodi type results for viscosity and classical solutions of nonlinear Dirichlet problems
for the fractional Laplace and comparable operators. In the choice of nonlinearities we consider semi-linear and 
super-linear growth cases separately. We develop a new technique using a functional integration-based approach, 
which is more robust in the non-local context than a purely analytic treatment.
\end{abstract}

\maketitle

\section{Introduction and statement of results}
In this paper our goal is to present a counterpart for the fractional Laplacian and operators comparable in a specific 
sense, of the classical Ambrosetti-Prodi problem
studied for a class of elliptic differential operators with nonlinear terms. In contrast with topological and variational
methods used in the classical context, we propose a new technique based on a path integration approach, which accommodates
a large class of non-local operators going well beyond the fractional Laplacian, and also applies to viscosity solutions.
This larger class is motivated by a number of applications including operators with L\'evy jump kernels having a lighter
than polynomial tail, however, an extension to this class requires a number of extra steps and concepts, which will be
pursued in a future work. Apart from this, another general advantage of our approach seems to be that it is more robust
than purely analytic techniques, dealing better with the difficulties resulting from boundary roughness. Our techniques
and framework have been developed recently in \cite{BL17a,BL17b}, to which we now intend to add the new dimension of
including nonlinearities. First we briefly recall the original problem, then state our results, and in the next section
present the proofs.

Let $\cD \subset \Rd$ be a bounded open domain with a $\cC^{2,\alpha}(\cD)$ boundary, $\alpha \in (0,1)$, and  consider
the Dirichlet problem
\begin{equation}
\label{APeq}
\left\{\begin{array}{ll}
\Delta u + f(u) = g(x) \quad \text{in}\, \; \cD\,, \\
\qquad \qquad u = 0\, \;\qquad \text{on}\,\; \partial \cD,
\end{array} \right.
\end{equation}
where $\Delta$ is the Laplacian, $f\in \cC^2(\RR)$, and $g \in \cC^{0,\alpha}(\bar\cD)$. In the pioneering paper \cite{AP72}
Ambrosetti and Prodi studied the operator $L = \Delta + f(\cdot)$ as a differentiable map between $C^{2,\alpha}(\bar\cD)$
and $\cC^{0,\alpha}(\bar\cD)$, and discovered the following phenomenon. Let $\lambda_1 < \lambda_2  \leq \lambda_3 \leq ...$
denote the Dirichlet eigenvalues of the Laplacian for the domain $\cD$. The authors have shown that provided $f$ is strictly
convex, with $f(0)=0$, and
\begin{equation}
\label{APcondi}
0 < \lim_{z\to -\infty} \frac{f(z)}{z} < \lambda_1 < \lim_{z\to\infty} \frac{f(z)}{z} < \lambda_2,
\end{equation}
then
\begin{enumerate}
\item[(1)]
there is a closed connected manifold $\sM_1 \subset \cC^{0,\alpha}(\bar\cD)$ of codimension 1, with the
property that there exist $\sM_0, \sM_2$ such that $\cC^{0,\alpha}(\bar\cD) \setminus \sM_1 = \sM_0 \sqcup \sM_2$,
\item[(2)]
the Dirichlet problem \eqref{APeq} has no solution if $g \in \sM_0$, has a unique solution if $g \in \sM_1$, and
has exactly two solutions if $g \in \sM_2$.
\end{enumerate}
The problem formulates in the wider context of invertibility of differentiable maps between Banach spaces,
in fact, $\sM_1$ is the set of elements $u$ on which the Fr\'echet derivative of $L$ is not locally invertible. Also, as
it is seen from condition \eqref{APcondi}, this split behaviour shows that the existence and multiplicity of
solutions is conditioned by the crossing of the nonlinear term with the principal eigenvalue of the linear part.

Following this fundamental observation, much work has been done in the direction of relaxing the conditions or
generalizing to further non-linear partial differential equations or systems. A first contribution has been made
by Berger and Podolak proposing a useful reformulation of the problem. Write
$$
L_1u = \Delta u + \lambda_1 u, \quad f_1(u) = f(u) - \lambda_1u, \quad g = \rho \varphi_1 + h,
$$
where $\varphi_1$ is the principal eigenfunction of the Dirichlet Laplacian, $h$ is in the orthogonal complement
of $\varphi_1$, $L^2$-normalized to 1, and $\rho \in \RR$, so that \eqref{APeq} becomes
\begin{equation}
\label{APequiv}
L_1u + f_1u = \rho \varphi_1 + h  \;\; \mbox{in} \;\; \cD, \quad u = 0 \;\; \mbox{in} \;\; \partial\cD.
\end{equation}
In \cite{BP75} it is then shown that there exists $\rho^*(h) \in \RR$, continuously dependent on $h$, such
that for $\rho > \rho^*(h)$ the equivalent Dirichlet problem has no solution, for $\rho = \rho^*(h)$ it has a
unique solution, and for $\rho < \rho^*(h)$ it has exactly two solutions. For further early developments we
refer to the works of Kazdan and Warner \cite{KW75} relaxing the assumptions, Dancer \cite{D78}, Amann and Hess
\cite{AH79} identifying a suitable growth condition on $f_1$, and Ruf and Srikanth \cite{RS86} turning to the
super-linear case. More recent papers exploring different perspectives include \cite{B81,FS84,MRZ,DFS,CTZ,FQS,SZ18,S14},
and for useful surveys we refer to de Figueiredo \cite{DF80} and Mahwin \cite{M87}. For non-local Hamilton-Jacobi
equations see \cite{DQT}, and for systems of non-local equations \cite{P17}.

Let $\cD \subset \Rd$ be a bounded domain with $\cC^2$ boundary, $s\in(0,1)$, and consider the fractional Laplacian-like
operator
$$L u(x)=\int_{\Rd} (u(x+y)-u(x)-\grad u(x)\cdot y\Ind_{\{\abs{y}\leq 1\}}) \frac{k(y/|y|)}{|y|^{d+2s}}\, dy,$$
where $k:\mathbb{S}^{d-1}\to(0, \infty)$ is a symmetric (i.e., $k(z)=k(-z)$), Borel-measurable function satisfying the
non-degeneracy  condition
$$
0< \Lambda_1\leq k(z)\leq \Lambda_2, \quad \text{for all}\; z\in \mathbb{S}^{d-1}.
$$
As it can be seen from Lemma~\ref{L3.8} below, our proof techniques use fine boundary behaviour of the solutions 
of the Dirichlet problem. It is known from \cite{RS16} that such a behaviour may not hold for a general non-degenerate 
kernel $k$ defined on $\Rd$. This is the reason why we restrict ourselves to the kernel functions of above type.

Motivated by the problem \eqref{APequiv}, in this paper we are interested in the existence and multiplicity
of solutions of
\begin{equation}\label{E-AP}
\tag{$P_\rho$}
\left\{\begin{array}{ll}
Lu + f(x, u) + \rho \Phi_1 + h(x) = 0\, \quad \text{in}\, \; \cD, \\
\qquad \qquad \qquad \qquad \qquad \quad u = 0\,  \quad \text{in} \, \;  \cD^c,
 \end{array} \right.
\end{equation}
where $\Phi_1$ is the Dirichlet principal eigenfunction of $L$ in $\cD$, $\rho\in\RR$, and $h\in\cC^\alpha(\bar\cD)$
for some $\alpha>0$. We also assume that $\norm{\Phi_1}_\infty=1$. Below we will consider viscosity solutions, however,
we will also discuss a sufficient condition on $f$ so that every viscosity solution becomes a classical solution.

Let $V\in \cC(\bar\cD)$, which will be referred to as a potential. We use the notation $\cC_{\rm b, +}(\Rd)$ for the
space of non-negative bounded continuous functions on $\Rd$. Also, we denote by $\cC^{2s+}(\cD)$ the space of continuous
functions on $\cD$ with the property that if $\psi\in \cC^{2s+}(\cD)$, then for every compact subset $\cK \subset \cD$
there exists $\gamma>0$ with $f\in \cC^{2s+\gamma}(\cK)$. Define
$$
\sF(\lambda, \cD) = \left\{\psi\in \cC_{\rm b, +}(\Rd)\cap\cC^{2s+}(\cD)\; :\; \psi>0\; \text{in}\; \cD, \; \text{and} \;
L\psi - V\psi + \lambda\psi \leq 0\right\}.
$$
The principal eigenvalue of $-L+V$ is defined as
\begin{equation}\label{E1.4}
\lambda^*(-L+V)=\sup\left\{\lambda\; :\; \sF(\lambda, \cD)\neq\emptyset\right\}.
\end{equation}
For easing the notation, we will simply write $\lambda^*$ for the above. This widely used characterization of the principal
eigenvalue originates from the seminal work of Berestycki, Nirenberg and Varadhan \cite{BNV}. Descriptions in a similar
spirit for a different class of non-local Schr\"{o}dinger operators have been obtained in \cite{BCV}, while in \cite{B18,DQT,SQX}
non-local Pucci operators have been considered. Recently, we proposed in \cite{BL17b} a probabilistic approach using a
Feynman-Kac representation to establish characterizations of the principal eigenvalue and the corresponding semigroup solutions.

Our first result concerns the existence of the principal eigenfunction and of a solution of the Dirichlet problem.
\begin{theorem}\label{T2.1}
Suppose that $V, g\in \cC^\alpha(\bar\cD)$ for some $\alpha>0$. The following hold:
\begin{itemize}
\item[(a)]
There exists a unique $\Psi_1\in\cC^{2s+}(\cD)\cap\cC_{\rm b, +}(\Rd)$, $\norm{\Psi_1}_\infty =1$, satisfying
\begin{equation}
\label{ET2.1A}
-L\Psi_1 + V\Psi_1 = \lambda^*\Psi_1 \;\; \mbox{in}\;\; \cD, \quad \Psi_1>0 \;\; \mbox{in}\;\; \cD, \quad
\Psi_1=0 \;\; \mbox{in}\;\; \cD^c.
\end{equation}
\item[(b)]
Suppose $\lambda^*>0$. Then there exists a unique $u\in \cC^{2s+}(\cD)\cap\cC(\Rd)$ satisfying
\begin{equation}
\label{ET2.1B}
-L u + Vu = g \;\; \mbox{in}\;\; \cD,  \quad u=0 \;\; \mbox{in}\;\; \cD^c.
\end{equation}
\end{itemize}
\end{theorem}

We will also need the following refined weak maximum principle for viscosity solutions.
\begin{theorem}\label{T2.2}
Suppose that $V\in\cC^\alpha(\bar\cD)$ and $\lambda^*>0$. Let $u\in\cC_{\rm b}(\Rd)$ be a viscosity subsolution of
$-L u + Vu\leq 0$ and $v\in\cC_{\rm b}(\Rd)$ be a viscosity supersolution of  $-L v + V v\geq 0$ in $\cD$.
Furthermore, assume that $u\leq v$ in $\cD^c$. Then $u\leq v$ in $\Rd$.
\end{theorem}

\begin{remark}
{\rm
On completion of this paper we have learnt that \cite{DQT} obtained results similar to Theorems~\ref{T2.1}-\ref{T2.2},
using a different technique than ours. In our understanding, these and related methods in the literature, applied also
for other purposes, depend on the comparability of the used non-local operators with the Riesz kernel and the fractional
Laplacian. We emphasize that in this paper we develop a path integration-based approach which, like in the framework
first set in \cite{BL17a,BL17b,B18}, is applicable for a large class of non-local operators (Markov generators of
L\'evy processes) without a similar restriction, also covering qualitatively different jump kernels. Since this will
need further probabilistic machinery, it will be presented elsewhere, and we limit ourselves to the fractional Laplacian
here. A recent paper \cite{B20} deals with a class of non-local operators and extends some of results of this paper using
a probabilistic framework.
}
\end{remark}

Next we impose the following Ambrosetti-Prodi type condition on $f$.

\medskip
\begin{assumption}\label{Assump-AP}
Let $f:\bar\cD\times\RR\to\RR$ be such that
\begin{enumerate}
\item[(1)]
$f$ is H\"{o}lder continuous in $x$, locally with respect to $u$, and locally Lipschitz continuous in $u$,
uniformly in $x\in\bar\cD$,
\medskip
\item[(2)]
there exist $V_1, V_2\in\cC^\alpha(\bar\cD)$, for some $\alpha>0$, such that
\begin{align}
\lambda^*(-L-V_1) &>0 \quad\text{and}\quad \lambda^*(-L-V_2)<0\,, \label{AP1}
\\[2mm]
f(x, q) &\geq V_1(x)q-C\quad \text{for all}\; q\leq 0, \; x\in\bar\cD\,, \label{AP2}
\\[2mm]
f(x, q) &\geq V_2(x)q-C\quad \text{for all}\; q\geq 0, \; x\in\bar\cD\,, \label{AP3}
\end{align}
\item[(3)]
$f$ has at most linear growth, i.e., there exists a constant $C > 0$ such that
$$
\abs{f(x,q)} \leq C(1+\abs{q}),
$$
for all $(x,q)\in\bar\cD\times\RR$, or
\medskip
\item[(3')]
$L=-\fL$ (i.e. $k$ is constant), $d>1+2s$ and there exists a positive continuous function $a_0$ such that
$$
\lim_{q\to\infty}\, \frac{f(x, q)}{q^p}= a_0(x), \quad \text{for some $p\in\left(1, \frac{d+2s}{d-2s}\right)$},
$$
where the above limit holds uniformly in $x\in\bar\cD$.
\end{enumerate}
\end{assumption}

\medskip
\noindent
When referring to Assumption [AP] below, we will understand that conditions (1), (2) and one of (3) or (3') hold.
In what follows, we assume with no loss of generality that $f(x,0)=0$, otherwise $h$ can be replaced by $h-
f(\cdot,0)$.

Now we are ready to state our main result on the fractional Ambrosetti-Prodi problem.
\begin{theorem}\label{T-AP}
Let Assumption [AP] hold. Then there exists $\rho^* = \rho^*(h) \in\RR$ such that for $\rho < \rho^*$ the
Dirichlet problem \eqref{E-AP} has at least two solutions, at least one solution for $\rho = \rho^*$, and no
solution for $\rho > \rho^*$.
\end{theorem}
\noindent
To prove our main Theorem~\ref{T-AP}, like in classic proofs such as in \cite{DF80, DFS}, in our context too the
viscosity characterization of the principal eigenfunction plays a key role. In Theorem~\ref{T2.1} therefore first
we obtain such a characterization. The refined maximum principle shown in Theorem \ref{T2.2} will also be a key
object towards the proof of the fractional Ambrosetti-Prodi phenomenon. We will rely on our recent work \cite{BL17b},
in which we proposed a method based on Feynman-Kac representations to establish Aleksandrov-Bakelman-Pucci (ABP)
estimates for semigroup solutions of non-local Dirichlet problems for a large class of operators, including but
going well beyond the fractional Laplacian. We will also show that every classical solution in our context here
is also a semigroup solution and thus a generalized ABP estimate can be established for these solutions, which
will be essential for obtaining the a priori estimates.

\section{Proofs}\label{S-proofs}
\subsection{Preliminaries}
We begin by recalling some notations and results from \cite{BL17a, BL17b}, which will be used below.
Let $(\Omega,\cF,\Prob)$ be a complete probability space, and $\pro X$ be an isotropic L\'evy process on this space
with infinitesimal generator $L$. Given a function $V\in\cC(\bar\cD)$, called potential, the corresponding Feynman-Kac
semigroup is given by
$$
T^{\cD, V}_t f(x)= \Exp^x\left[e^{-\int_0^t V(X_s)\, \D{s}} f(X_t)\Ind_{\{t<\uptau_\cD\}}\right], \quad t>0,
\; x\in\cD, \; f\in L^2(\cD),
$$
where
$$
\uptau_\cD = \inf\{t > 0: \, X_t \not\in \cD\}
$$
is the first exit time of the process $\pro X$ from the domain $\cD$. When $L=-\fL$, $0 < s < 1$, 
it is shown in \cite[Lem~3.1]{BL17a} that
$T^{\cD, V}_t$, $t > 0$, is a Hilbert-Schmidt operator on $L^2(\cD)$ with continuous integral kernel in
$(0, \infty)\times\cD\times\cD$. Moreover, every operator $T_t$ has the same purely discrete spectrum, independent
of $t$, whose lowest eigenvalue is the principal eigenvalue $\lambda^*$ having multiplicity one, and the corresponding
principal eigenfunction $\Psi \in L^2(\cD)$ is strictly positive. We also have from \cite[Lem.~3.1]{BL17a} that $\Psi
\in\cC_0(\cD)$, where $\cC_0(\cD)$ denotes the class of continuous functions on $\Rd$ vanishing in $\cD^c$. Since $\Psi$
is an eigenfunction in semigroup sense, we have for all $t>0$ that
\begin{equation}\label{E3.1}
e^{-\lambda^* t} \Psi(x) = T_t \Psi(x) = \Exp^x\left[e^{-\int_0^t V(X_s)\, \D{s}} \Psi(X_t)\Ind_{\{t<\uptau_\cD\}}\right],
\quad x\in\cD.
\end{equation}

Let $\seq\cD$ be a collection of strictly decreasing domains with the property that $\cap_{n\geq 1}\cD_n=\cD$,
and each $\cD_n$ having its boundary satisfying the exterior cone condition. Denote by $\lambda^*_n$ the principal
eigenvalue in sense of \eqref{E1.4}. The following result  will be useful below (see also, \cite[Lem.~4.2]{BL17b}).
\begin{proposition}\label{P3.1}
The following hold:
\begin{itemize}
\item[(1)]
For every $n\in\NN$ we have $\lambda^*>\lambda^*_n$ and $\lim_{n\to\infty} \lambda^*_n=\lambda^*$.
\item[(2)]
Let $\tilde V\geq V$ and suppose that for an open set $\cU\subset\cD$ we have $\tilde V> V$ in $\cU$. Then
$\lambda^*_{\tilde V}>\lambda^*_V$, where $\lambda^*_V$ and $\lambda^*_{\tilde V}$ denote the principal
eigenvalues corresponding to the potentials $V$ and $\tilde{V}$, respectively.
\end{itemize}
\end{proposition}

\begin{proof}
Existence of a unique principal eigenfunction follows from \cite[Th.~1.1]{SQX} (see also \cite{DQT}). Note that \cite{SQX} considers the case $s>\frac{1}{2}$ due to the presence of the drift
term and the same proof would go through in our setting. Using \cite[Th.~1.3]{Serra15} we can show that the eigenfunction belongs to $\cC^{2s+}(\cD)$. Then the strict monotonicity
of the eigenvalue with respect to domains follows from \cite[Theorem~5.1]{SQX}. Using the arguments of \cite[Th.~1.6]{B18} it can be shown that $\lim_{n\to\infty} \lambda^*_n=\lambda^*$.
Part (2) again follows from \cite[Th.~5.1]{SQX}.
\end{proof}

Since the theory developed in \cite{BL17a} is probabilistic while here we are concentrating on viscosity
solutions, we point out the relationship between these notions of solution (compare also with \cite[Rem.~3.2]{BL17b}, \cite[Sec.~3.1]{B19}).
We say that $u\in\cC_{\rm b}(\Rd)$ is a semigroup sub-solution  of
$$
-L u+V u \leq g \quad \mbox{in $\cD$},
$$
if we have for all $x\in \cD$ that
$$
u(x) \leq \Exp^x\left[e^{-\int_0^{t\wedge\uptau_\cD} V(X_s) \D{s}} u(X_{t\wedge\uptau_\cD})\right]
+ \Exp^x\left[\int_0^{t\wedge\uptau_\cD} e^{-\int_0^{s} V(X_p) \D{p}} g(X_{s})\, \D{s}\right], \quad t\geq 0.
$$
Semigroup super-solutions are defined in an analogous way.

\begin{lemma}\label{L3.1}
Suppose that $V, g\in\cC(\bar\cD)$, and let $u$ satisfy
\begin{equation}
\label{flu}
-L u+V u \leq g \quad \mbox{in $\cD$}.
\end{equation}
We have the following:
\begin{itemize}
\item[(1)]
If $u\in\cC_{\rm b}(\Rd)$ is a semigroup sub-solution (resp., super-solution) of \eqref{flu}, then it is also a viscosity
sub-solution (resp., super-solution).
\item[(2)]
If $u\in\cC^{2s+}(\cD)\cap\cC_{\rm b}(\Rd)$ is a classical sub-solution (super-solution) of \eqref{flu}, then it is also a
semigroup sub-solution (resp., super-solution).
\end{itemize}
\end{lemma}

\begin{proof}
Consider part (1). Choose a point $x\in\cD$, and let $\varphi\in \cC^2(\cD)$ be a test function that (strictly)
touches $u$ at $x$ from above, i.e., for a ball $\cB_r(x)\subset \cD$ we have $\varphi(x)=u(x)$, and $\varphi(y)>
u(y)$ for $y\in \cB_r(x)\setminus\{x\}$. Define
\[
\varphi_r(y)=\left\{\begin{array}{lll}
\varphi(y) & y\in \cB_r(x), \\
u(y) & y\in \cB^c_r(x).
\end{array} \right.
\]
To show that $u$ is viscosity solution, we need to show that $-L\varphi_r(x) + V(x) u(x)\leq g(x)$. Consider a sequence
of functions $(\varphi_{r, n})_{n\in\mathbb N} \subset \cC^2(\cB_r(x))\cap \cC(\Rd)$ with the property that $\varphi_{2, n}=
\varphi_r$ outside $\cB_{r+\frac{1}{n}}(x)\setminus \cB_r(x)$, $\varphi_{r, n}\geq u$, and $\varphi_{r, n}\to \varphi_r$ almost
surely, as $n\to\infty$. Since $u$ is a semigroup subsolution, we have that
$$
u(x) \leq \Exp^x\left[e^{-\int_0^{t\wedge\uptau_\cD} V(X_s) \D{s}} u(X_{t\wedge\uptau_\cD})\right]
+ \Exp^x\left[\int_0^{t\wedge\uptau_\cD} e^{-\int_0^{s} V(X_p) \D{p}} g(X_{s})\, \D{s}\right], \quad t\geq 0.
$$
It is direct to show that $\pro Y$, $Y_t= e^{-\int_0^{t\wedge\uptau_\cD} V(X_s) \D{s}} u(X_{t\wedge\uptau_\cD})
+ \int_0^{t\wedge\uptau_\cD} e^{-\int_0^{s} V(X_p) \D{p}} g(X_{s})\, \D{s}$, is a submartingale with respect to the natural
filtration of $(X_{t\wedge\uptau_\cD})_{t\geq 0}$, see also \cite{BL17b}, hence by optional sampling we obtain that
\begin{equation}\label{ET3.1A}
u(x) \leq \Exp^x\left[e^{-\int_0^{t\wedge\uptau_r} V(X_s) \D{s}} u(X_{t\wedge\uptau_r})\right]
+ \Exp^x\left[\int_0^{t\wedge\uptau_r} e^{-\int_0^{s} V(X_p) \D{p}} g(X_{s})\, \D{s}\right], \quad t\geq 0,
\end{equation}
where $\uptau_r$ denotes the first exit time from the ball $\cB_r(x)$. On the other hand, by applying It\^{o}'s formula
on $\varphi_{r, n}$ we obtain
$$
\Exp^x\left[e^{-\int_0^{t\wedge\uptau_r} V(X_s) \D{s}} \varphi_{r, n}(X_{t\wedge\uptau_\cD})\right]-\varphi_{r, n}(x)
= \Exp^x\left[\int_0^{t\wedge\uptau_r} e^{-\int_0^{s} V(X_p)\D{p}} (L\varphi_{r, n}-V\varphi_{r, n})(X_{s})\, \D{s}\right],
$$
for all $t\geq 0$. Combining this with \eqref{ET3.1A} gives
$$
\Exp^x\left[\int_0^{t\wedge\uptau_r} e^{-\int_0^{s} V(X_p)\D{p}} (L\varphi_{r, n}-V\varphi_{r, n})(X_{s})\, \D{s}\right]
+\Exp^x\left[\int_0^{t\wedge\uptau_r} e^{-\int_0^{s} V(X_p) \D{p}} g(X_{s})\, \D{s}\right] \geq 0.
$$
On dividing both sides by $t$ and letting $t\to 0$, it follows that
$$
L\varphi_{r, n}(x)-V(x)\varphi_{r, n}(x) + g(x)\geq 0.
$$
Thus by letting $n\to\infty$, we obtain
$$
-L\varphi_{r}(x)+V(x)\varphi_{r}(x) \leq  g(x)\,,
$$
which proves the first part of the claim.

Next consider part (2). By the property of $u$ we note that $L u$ is continuous in $\cD$. Consider a sequence
of open sets $\cK_n\Subset K_{n+1} \Subset \cD$ and $\cup_n K_n=\cD$. For fixed $n$, let $(\psi_m)_{m\in\mathbb N}
\subset \cC^2(\cD)\cap\cC_{\rm b}(\Rd)$ be a sequence of functions satisfying
$$
\sup_{x\in\bar{\cK}_n}\abs{L u(x)-L\psi_m(x)} + \sup_{x\in\Rd}\abs{u(x)-\psi_m(x)}\to 0,\quad
\text{as}\quad  m\to\infty.
$$
Applying It\^{o}'s formula to $\psi_m$, we get that
$$
\Exp^x\left[e^{-\int_0^{t\wedge\uptau_n} V(X_s) \D{s}} \psi_{m}(X_{t\wedge\uptau_n})\right]-\psi_{m}(x)
= \Exp^x\left[\int_0^{t\wedge\uptau_n} e^{-\int_0^{s} V(X_p)\D{p}} (L\varphi_{m}-V\varphi_{m})(X_{s})\, \D{s}\right],
$$
where $\uptau_n$ denotes the first exit time from the set $\cK_n$. First letting $m\to\infty$ and then $n\to\infty$
above, and using the fact that $\uptau_n\uparrow \uptau_\cD$ almost surely, we obtain
$$
\Exp^x\left[e^{-\int_0^{t\wedge\uptau_\cD} V(X_s) \D{s}} u(X_{t\wedge\uptau_n})\right]-u(x)
\geq  -\Exp^x\left[\int_0^{t\wedge\uptau_\cD} e^{-\int_0^{s} V(X_p)\D{p}} g (X_{s})\, \D{s}\right],\quad t\geq 0.
$$
This shows that $u$ is a semigroup subsolution.
\end{proof}

\subsection{Proof of Theorem~\ref{T2.1}}
Now we are ready to prove our first theorem.
\begin{proof}
First consider (a). As discussed in Proposition~\ref{P3.1}, there exists an
eigenpair $(\lambda^*, \Psi)\in\RR\times\cC_0(\cD)$ with $\Psi>0$ in $\cD$, satisfying
\begin{equation}\label{ET2.1C}
-L\Psi=(\lambda^*-V)\Psi \quad \text{in}\; \cD, \quad \text{and}\quad \Psi=0\quad \text{in}\; \cD^c,
\end{equation}
in viscosity sense.
By \cite[Th.~2.6]{B18} we have $\Psi\in\cC^\alpha(\Rd)$ for some $\alpha>0$, independent of $\Psi$. Since $V$ is H\"{o}lder continuous, it follows that $(\lambda^*-V)\Psi$
is H\"{o}lder continuous in $\bar\cD$. A combination of \eqref{ET2.1C} and \cite[Th.~1.3]{Serra15} gives that $\Psi\in
\cC^{2s+}(\cD)$, implying existence for \eqref{ET2.1A}. Using Lemma~\ref{L3.1} we also note that $(\lambda^*, \Psi)$ satisfies \eqref{E3.1}. Simplicity of
$\lambda^*$ again follows from \cite[Th.~1.2]{SQX}.

Next we consider (b). The main idea in proving \eqref{ET2.1B} is to use Schauder's fixed point theorem. Consider a map
$\cT:\cC_0(\cD)\to\cC_0(\cD)$ defined such that for every $\psi\in\cC_0(\cD)$, $\cT\psi=\varphi$ is the unique viscosity
solution of
\begin{equation}\label{ET2.1E}
-L\varphi =g-V\psi\quad \text{in}\; \cD, \quad \text{and}\quad \varphi=0\quad \text{in}\; \cD^c.
\end{equation}
Using \cite[Th.~2.6]{B18} we know that
$$
\norm{\cT\psi}_{\cC^s(\Rd)}\leq c_1 (\norm{g}_\infty + \norm{V\psi}_\infty),
$$
for a constant $c_1 = c_1(\cD, d, s)$. This implies that $\cT$ is a compact linear operator. It is also easy to see that
$\cT$ is continuous.

In a next step we show that the set
$$
\sB= \big\{\varphi\in\cC_0(\cD)\; :\; \varphi=\mu\cT\varphi \;\; \text{for some}\;\; \mu\in[0, 1]\big\}
$$
is bounded in $\cC_0(\cD)$. For every $\varphi\in\sB$ we have
\begin{equation}\label{ET2.1F}
-L\varphi = \mu g - \mu V\varphi\quad \text{in}\; \cD, \quad \text{and}\quad \varphi=0\quad \text{in}\; \cD^c,
\end{equation}
for some $\mu\in[0,1]$. As argued above, we note that $\varphi\in\cC^{2s+}(\cD)\cap\cC_0(\cD)$.
Thus by Lemma~\ref{L3.1} we see that $\varphi$ is a semigroup solution of \eqref{ET2.1F}. To show boundedness of $\sB$
it suffices to show that for a constant $c_2$, independent of $\mu$, we have
\begin{equation}\label{ET2.1G}
\sup_{x\in\bar\cD}\abs{\varphi(x)}\leq c_2\, \sup_{x\in\bar\cD}\abs{g(x)}.
\end{equation}
Once \eqref{ET2.1G} is established, the existence of a fixed point of $\cT$ follows by Schauder's fixed point theorem.
Since every solution of \eqref{ET2.1B} is a semigroup solution and $\lambda^*>0$, the uniqueness of the solution follows
from \cite[Th.~4.5]{BL17b}. To obtain \eqref{ET2.1G} recall from \cite[Cor.~4.3]{BL17b} (which basically uses Proposition~\ref{P3.1} above) that
\begin{equation}\label{ET2.1H}
\lambda^*_{\mu V}=
-\lim_{t\to\infty} \frac{1}{t} \log \Exp^x\left[e^{-\int_0^t \mu V(X_s)\, \D{s}}\Ind_{\{\uptau_\cD>t\}}\right], \quad x\in\cD\,.
\end{equation}
Let $\lambda^*_0>0$ be the principal eigenvalue corresponding to the potential $V=0$. Then from the concavity of the map
$\mu\mapsto \lambda^*_{\mu V}$, which results from \eqref{ET2.1H} by applying Young's inequality, it follows that
$$
\lambda^*_{\mu V}\geq \lambda^*_V \wedge \lambda^*_0 = 2\delta>0.
$$
Hence by using \eqref{ET2.1H} and the continuity of $\mu\mapsto\lambda^*_{\mu V}$, we find constants $c_3>0, \mu_0>1$, such
that for every $\mu\in[0, \mu_0]$ we have
\begin{equation}\label{ET2.1I}
\Exp^x\left[e^{-\int_0^t \mu V(X_s)\, \D{s}}\Ind_{\{\uptau_\cD>t\}}\right]\leq c_3 e^{-\delta t},\quad t\geq 0, \; x\in\cD.
\end{equation}
Since $\varphi$ is a semigroup solution, we have that
$$
\varphi(x) = \Exp^x\left[ e^{-\int_0^t \mu V(X_s)\, \D{s}} \varphi(X_t)\Ind_{\{\uptau_\cD>t\}}\right] +
\int_0^t T^{\cD, \mu V}_s g (x) \D{s}.
$$
Letting $t\to\infty$, using \eqref{ET2.1I} and H\"{o}lder inequality, it is easily seen that the first term at the right hand
side of the above vanishes. Again by \eqref{ET2.1I}, we have for $x\in\cD$
$$
\Big|T^{\cD, \mu V}_t g (x)\Big| \leq c_3 \sup_{x\in\bar\cD}\abs{g}\, e^{-\delta t}, \quad t\geq 0.
$$
Thus finally we obtain
$$
\sup_{x\in\bar\cD}\abs{\varphi(x)}\leq\; \frac{c_3}{\delta} \sup_{x\in\bar\cD}\abs{g(x)},
$$
yielding \eqref{ET2.1G}.
\end{proof}

\subsection{Proof of Theorem~\ref{T2.2}}
First we show that the $\cC^2$-class of test functions can be replaced by functions of class $\cC^{2s+}$ in the definition of
the viscosity solution.
\begin{lemma}\label{L3.2}
Let $u\in \cC_{\rm b}(\Rd)$ be a viscosity subsolution of $\fL u + V u\leq g$ in $\cD$. Consider $x\in\cD$. Suppose that
there exists an open set  $\cN\Subset \cD$, containing $x$, and a function $\varphi\in \cC^{2s+}(\bar \cN)$ satisfying
$\varphi(x)=u(x)$ and $\varphi> u$ in $\cN\setminus\{x\}$. Define
\[
\varphi_\cN(y)=\left\{\begin{array}{ll}
\varphi(y) & \text{for}\;\; y\in \cN,
\\[2mm]
u(y) & \text{for}\;\; y\in \Rd\setminus \cN.
\end{array} \right.
\]
Then we have $-L\varphi_\cN(x) + V(x) u(x)\leq g(x)$.
\end{lemma}

\begin{proof}
Consider a sequence of functions in $(\varphi_m)_{m\in\mathbb N}$, $\cC^2$ in a neighbourhood of $x$, and such that
$\norm{\varphi_m-\varphi}_{\cC^{2s+\alpha}(\bar N)}\to 0$, for some $\alpha>0$, as $m\to\infty$. This is possible since
$\varphi\in\cC^{2s+}(\bar N)$. Let
$$
\delta_m = \min_{\bar N}(\varphi_m-u).
$$
Then $\hat{\varphi}_m=\varphi_m-\delta_m$ touches $u$ from above in $\bar N$. Since $\sup_{\bar N}\abs{\hat{\varphi}_m-u}
\to 0$, it follows that there exists a sequence $(x_m)_{m\in\mathbb N}\in N$ such that $x_m\to x$, $\delta_m\to 0$, as $m\to\infty$, and
$\hat{\varphi}_m(x_m)=u(x_m)$. Set
\[
\varphi_{N,m}(y)=
\left\{\begin{array}{ll}
\hat{\varphi}_m(y) & \text{for}\;\; y\in N,
\\[2mm]
u(y) & \text{for}\;\; y\in \Rd\setminus N.
\end{array}
\right.
\]
By the definition of the viscosity subsolution we find
\begin{align*}
& -\frac{C(d, s)}{2}\int_{B_r(x)}\frac{\varphi_{N,m}(x_m+y)+\varphi_{N,m}(x_m-y)-2\varphi_{N,m}(x_m)}{\abs{y}^{d+2s}}\, \D{y}\\
& - \frac{C(d, s)}{2}\int_{B_r^c(x)}\frac{\varphi_{N,m}(x_m+y)+\varphi_{N,m}(x_m-y)-2\varphi_{N,m} (x_m)}{\abs{y}^{d+2s}}\, \D{y}
+ V(x_m) u(x_m)
\leq g(x_m),
\end{align*}
where $C(d, s)$ is the normalizing constant for fractional Laplacian and $r>0$ is chosen to satisfy $B_{2r}(x)\Subset N$.
 It is easily seen that we can let $m\to\infty$ above
and use the continuity of $V, g, u$ to obtain
$$
-L \varphi_N(x) + V(x) u(x)\leq g(x),
$$
which shows the claim.
\end{proof}

Next we prove our second theorem stated in the previous section.
\begin{proof}[Proof of Theorem~\ref{T2.2}]
Let $w=u-v$. By \cite[Th.~5.9]{CS09} it then follows that
\begin{equation}\label{ET2.2A}
-L w + V w\leq 0 \quad \text{in}\; \cD,
\end{equation}
in viscosity sense. Note that $w\leq 0$ in $\cD^c$, while we need to show that $w\leq 0$ in $\Rd$. Suppose, to the contrary,
that $w^+>0$ in $\cD$. Using Proposition \ref{P3.1}, we find a domain $\cD_1\Supset\cD$ with a $\cC^1$-boundary and
$\lambda^*_1>0$, where $\lambda^*_1$ is the principal eigenvalue for $\cD_1$ and potential $V$. In fact, we may take $V$ as
a $\cC^\alpha$-extension from $\cD$ to $\cD_1$. Let $\Psi_1\in\cC^{2s+}(\cD_1)\cap\cC_0(\cD_1)$ be the corresponding positive
principal eigenfunction. Thus we have
\begin{equation}\label{ET2.2B}
-L \Psi_1 + V \Psi_1=\lambda^*_1 \Psi_1 \;\; \text{in}\;\; \cD_1 \quad \text{and} \quad \Psi_1=0 \;\; \text{in}\;\; \cD^c_1.
\end{equation}
Define
$$
c_0=\inf\big\{c\in(0, \infty)\; :\; c\Psi_1-w>0 \;\; \text{in}\;\; \cD\big\}.
$$
Since $\min_{\cD}\Psi_1>0$, it follows that $c_0$ is finite, and $w^+>0$ implies that $c_0>0$. Then $\Phi=c_0\Psi_1-w$ necessarily
vanishes at some point, say $x_0\in\cD$. This follows from the fact that $w^+=0$ on $\partial\cD$. Thus $c_0\Psi_1$ lies above $u$
on all of $\Rd$ and touches $w$ at $x_0$. Hence by \eqref{ET2.2A} and Lemma~\ref{L3.2} it follows that
$$
-L(c_0\Psi_1)(x_0) + V(x_0)(c_0\Psi_1(x_0))\leq 0.
$$
This leads to a contradiction as the left hand side of the above expression equals $\lambda^*_1(c_0\Psi_1(x_0))>0$ by \eqref{ET2.2B}.
\end{proof}

\subsection{Proof of Theorem~\ref{T-AP}}
Now we turn to proving our main result on the fractional Ambrosetti-Prodi phenomenon. The strategy of proof will be divided
in the following steps.
\begin{itemize}
\item[(1)]
First we find $\rho_1$ such that for every $\rho\leq \rho_1$ there exists a minimal solution of \eqref{E-AP}. This will be done in
Lemmas~\ref{L3.3} and ~\ref{L3.4} below.
\item[(2)]
Next we find $\rho_2 > \rho_1$ such that no solution of \eqref{E-AP} above $\rho_2$ exists. This is the content of Lemma~\ref{L3.7}
and Lemma~\ref{L2.8}.
\item[(3)]
Finally, we follow the arguments in \cite{DF80} to find the bifurcation point $\rho^*$.
\end{itemize}

We begin by showing the existence of a sub/super-solution, which will be used for constructing a minimal solution.
\begin{lemma}\label{L3.3}
Let Assumption [AP] hold. The following hold:
\begin{itemize}
\item[(1)]
For every $\rho\in\RR$ there exists $\underline{u}\in \cC^{2s+}(\cD)\cap\cC_0(\cD)$ satisfying $\underline{u}\leq 0$ in $\cD$
and
$$
-L \underline{u} \leq f(x, \underline{u}) + \rho \Phi_1 + h(x) \quad \text{in}\;\; \cD.
$$
\item[(2)]
There exists $\bar{\rho}_1<0$ such that for every $\rho\leq \bar{\rho}_1$ there exists $\bar{u}\in \cC^{2s+}(\cD)\cap\cC_0(\cD)$
satisfying $\bar{u}\geq 0$ in $\cD$ and
$$
-L \bar{u} \geq f(x, \bar{u}) + \rho \Phi_1 + h(x) \quad \text{in}\;\; \cD.
$$
\item[(3)]
We can construct $\underline{u}$ to satisfy $\underline{u}\leq \hat{u}$, for every super-solution $\hat{u}$ of
$$
-L \hat{u} \geq f(x, \hat{u}) + \rho \Phi_1 + h(x) \quad \text{in}\;\; \cD,
$$
with $\hat{u}\in\cC^{2s+}(\cD)\cap\cC_0(\cD)$.
\end{itemize}
\end{lemma}

\begin{proof}
Consider $\rho\in\RR$. Let $C_2= 2\sup_{\bar\cD}\abs{h} + 2\abs{\rho} + C $, where $C$ is the same constant as in
\eqref{AP2}-\eqref{AP3}. Since $\lambda^*(-L-V_1)>0$ by \eqref{AP1}, it follows from Theorem~\ref{T2.1}(2) that
there exists a unique $\underline{u}\in\cC^{2s+}(\cD)\cap\cC_0(\cD)$ satisfying
\begin{equation}\label{EL3.3A}
-L\underline{u} - V_1 \underline{u}= -C_2+h(x)+ \rho\Phi_1\quad \text{in}\; \cD.
\end{equation}
Recalling that $\Phi_1\in \cC^\alpha(\Rd)\cap\cC_0(\cD)$ by \cite[Th.~2.6]{B18}, the right hand side of the \eqref{EL3.3A} is
H\"{o}lder continuous in $\cD$. By our choice of $C_2$ we see that
$$
-L\underline{u} - V_1 \underline{u}\leq 0,
$$
and hence, by Theorem~\ref{T2.2} we have $\underline{u}\leq 0$ in $\Rd$. Therefore, by making use of \eqref{AP2}
we get that
$$
-L\underline{u} \leq f(x, \underline{u}) + h(x)+ \rho\Phi_1\quad \text{in}\; \cD, \quad \text{and}\quad
\underline{u}=0\quad \text{in}\; \cD^c.
$$
This proves part (1).

Now we proceed to establish (2). Due to Assumption [AP] there exists a constant $C_1$ satisfying $f(x, q)\leq C_1 (1+q^p)$,
for all $(x, q)\in \bar\cD\times[0, \infty)$. We consider the unique function $\bar{u}\in\cC^{2s+}(\cD)\cap\cC_0(\cD)$
satisfying
\begin{equation}\label{EL3.3B}
-L\bar{u}-h^+-C_1=0\quad \text{in}\; \cD.
\end{equation}
Using \cite[Th.~2.6, Eqn (2.3)]{B18} we find $c_1 = c_1(d, s, \cD) > 0$, such that
\begin{equation}\label{EL3.3C}
\sup_{x\in\cD}\frac{\abs{\bar u(x)}}{d^s(x)}\leq c_1,
\end{equation}
where $d(\cdot)$ is the distance function from the boundary of $\cD$. Since $-L\bar{u}\geq 0$, it also follows from Hopf's
lemma \cite[Th.~2.4]{B18} that $\bar{u}>0$ in $\cD$. Since $-L\Phi_1=\lambda^*_0\Phi_1\geq 0$ in $\cD$, another application
of Hopf's lemma gives a constant $c_2>0$ satisfying
$$
\frac{\Phi_1(x)}{d^s(x)}\geq c_2, \quad x\in \cD.
$$
Combining the above with \eqref{EL3.3C} and choosing $-\bar\rho_1>0$ large, we find for every $\rho\leq \bar\rho_1$ that
$$
- \rho\Phi_1\geq C_1c^p_1 d^{sp}\geq C_1 \bar{u}^p, \quad \text{for}\; x\in\cD.
$$
Hence by \eqref{EL3.3B} we have for $\rho\leq \rho_0$
$$
-L\bar{u}\geq f(x, \bar{u}) + \rho\Phi + h\quad \text{in}\; \cD.
$$
This proves (2).

Now we come to (3). Note that
$$-L\hat{u}\geq f(x, \hat{u}) - \abs{\rho} - \norm{h}_\infty \quad \text{in}\; \cD.$$
Since the minimum of two viscosity super-solutions is again a viscosity super-solution, we note that $w= \hat{u} \wedge 0$
is a viscosity super-solution  of
\begin{equation}\label{EL3.3D}
-L w\geq f(x, w)- \abs{\rho} - \norm{h}_\infty \geq V_1w -C- \abs{\rho} - \norm{h}_\infty \quad \text{in}\; \cD,
\end{equation}
by \eqref{AP2}. On the other hand, by our choice of $C_2$ in \eqref{EL3.3A} we have
\begin{equation}\label{EL3.3E}
-L \underline{u}-V_1\underline{u}\leq -C -|\rho|-\norm{h}_\infty \quad \text{in}\; \cD.
\end{equation}
Combining \eqref{EL3.3D}, \eqref{EL3.3E} and \cite[Th.~5.9]{CS09}, we obtain
$$
-L(w-\underline{u}) - V_1(w-\underline{u})\geq 0\quad \text{in}\; \cD,
$$
in viscosity sense, and $w-\underline{u}=0$ in $\cD^c$. Hence by Theorem~\ref{T2.2} we have $w\geq \underline{u}$ in $\Rd$,
implying $\hat{u}\geq w\geq \underline{u}$ in $\Rd$. This yields part (3).
\end{proof}

Using Lemma~\ref{L3.3} we can now prove the existence of a minimal solution.
\begin{lemma}\label{L3.4}
For $\rho\leq \bar\rho_1$, where $\bar\rho_1$ is the same value as in Lemma~\ref{L3.3}, there exists $u\in\cC^{2s+}(\cD)
\cap\cC_0(\cD)$ satisfying
\begin{equation}\label{EL3.4A}
\fL u = f(x, u) + \rho \Phi_1 + h(x) \quad \text{in}\; \cD.
\end{equation}
Moreover, the above $u$ can be chosen to be minimal in the sense that if $\tilde u\in\cC^{2s+}(\cD)\cap\cC_0(\cD)$ is another
solution of \eqref{EL3.4A}, then $\tilde u \geq u$ in $\Rd$.
\end{lemma}

\begin{proof}
The proof is based on the standard monotone iteration method. Denote $m=\min_{\bar\cD} \underline{u}$ and
$M=\max_{\bar\cD}\bar{u}$. Let $\theta>0$ be a Lipschitz constant for $f(x, \cdot)$ on the interval $[m, M]$, i.e.,
$$
\abs{f(x, q_1)-f(x, q_2)}\leq \theta \abs{q_1-q_2}\quad \text{for}\; q_1, q_2\in [m, M], \; x\in\bar\cD.
$$
Denote $F(x, u)= f(x, u) + \rho\Phi(x)+h(x)$. Consider the solutions of the following family of problems:
\begin{equation}\label{EL3.4B}
\begin{split}
-L u^{(n+1)} + \theta u^{(n+1)} &= F(x, u^{(n)}) + \theta u^{(n)}\quad \text{in}\; \cD,
\\
u^{(n+1)} &= 0 \quad \text{in}\; \cD^c.
\end{split}
\end{equation}
By Theorem~\ref{T2.1}(2) equation \eqref{EL3.4B} has a unique solution, provided $u^{(n)}$ is H\"{o}lder continuous in $\bar\cD$. We
set $u^{(0)}=\underline{u}$. Since $u^{(0)}\in\cC^\alpha(\Rd)$ by \cite[Th.~2.6]{B18}, it follows from \cite{Serra15} that $u^{(1)} \in\cC^{2s+}(\cD)\cap\cC^\alpha(\Rd)$.
Thus by successive iteration it follows that $u^{(n)} \in\cC^{2s+}(\cD)\cap\cC^\alpha(\Rd)$, for all $n\geq 0$. Hence all solutions of
\eqref{EL3.4B} are classical solutions. Again, it is routine to check from \eqref{EL3.4B} and Theorem~\ref{T2.2} that $u^{(0)}\leq
u^{(n)}\leq u^{(n+1)}\leq \bar{u}$ in $\cD$. This implies $\sup_{\Rd}\abs{u^{(n)}}\leq M-m$, for all $n$. Thus applying
\cite[Th.~2.6]{B18} we obtain
$$
\sup_{n\in\mathbb N}\; \norm{u^{(n)}}_{\cC^\alpha(\Rd)}\leq \kappa_1,
$$
for some constants $\alpha, \kappa_1$. Hence there exists $u\in\cC^\alpha(\Rd)\cap\cC_0(\cD)$ such that $u^{(n)}\to u$ in $\cC_0(\cD)$ as $n\to\infty$.
Using the stability of viscosity solutions, it then follows that $u$ is a viscosity solution to
\begin{equation*}
\begin{split}
-L u &= F(x, u) \quad \text{in}\; \cD,
\\
u &= 0 \quad \text{in}\; \cD^c.
\end{split}
\end{equation*}
We can now apply the regularity estimates from \cite{Serra15} to show that $u\in\cC^{2s+}(\cD)$.

To establish minimality we consider a solution $\tilde u$ of \eqref{EL3.4A} in $\cC^{2s+}(\cD)\cap\cC_0(\cD)$.
From Lemma~\ref{L3.3}(3) we have $\underline{u}\leq \tilde u$ in $\Rd$. Thus $\bar{u}$ can be replaced by
$\tilde u$, and the above argument shows that $u\leq \tilde u$.
\end{proof}

Now we derive a priori bounds on the solutions of \eqref{E-AP}. Our first result bounds the negative part of solutions
$u$ of \eqref{E-AP}. We recall that under the standing assumptions on $f$, any viscosity solution of \eqref{E-AP} is an
element of $\cC^{2s+}(\cD)\cap\cC_0(\cD)$, and thus also a classical solution.

\begin{lemma}\label{L3.5}
Let Assumption [AP](2) hold. There exists a constant $\kappa = \kappa(d, s, \cD, V_1)$, such that for every solution $u$
of \eqref{E-AP} with $\rho \geq -\hat\rho$, $\hat\rho>0$, we have
$$
\sup_{\cD}\abs{u^-}\;\leq\; \kappa (C+\hat\rho+\norm{h}_\infty),
$$
where $C$ is the same constant as in \eqref{AP2}.
\end{lemma}

\begin{proof}
First observe that if $u$ is a solution to \eqref{E-AP} for some $\rho\geq -\hat\rho$, then
$$
L u + f(x, u)\leq \hat\rho + \norm{h}_\infty\quad \text{in}\; \cD.
$$
Defining $w=u \wedge 0$ we see that $w$ is a viscosity super-solution of the above equation, i.e.,
$$
L w + f(x, w)\leq \hat\rho + \norm{h}_\infty\quad \text{in}\; \cD, \quad \text{and}\quad w=0\quad \text{in}\; \cD^c.
$$
From \eqref{AP2} it then follows that
$$
L w + V_1 w\leq C+\hat\rho + \norm{h}_\infty\quad \text{in}\; \cD, \quad \text{and}\quad w=0\quad \text{in}\; \cD^c,
$$
in viscosity sense. Let $v\in\cC^{2s+}(\cD)\cap\cC_0(\cD)$ be the unique solution of
$$
L v + V_1 v= C+\hat\rho + \norm{h}_\infty\quad \text{in}\; \cD, \quad \text{and}\quad v=0\quad \text{in}\; \cD^c.
$$
Existence follows from Theorem~\ref{T2.1}(2). Applying Theorem~\ref{T2.2}, we get $-w\leq -v$ in $\Rd$. Since $v$ is also
a semigroup solution by Lemma~\ref{L3.1}, we obtain from \cite[Th.~4.12 and Rem.~3.9]{BL17b} that with a constant $\kappa = \kappa(s,d,\cD,V_1)$

$$
\sup_{x\in\bar\cD}\abs{v}\;\leq\; \kappa (C+\hat\rho + \norm{h}_\infty)
$$
holds. Thus $u^-=-w\leq \kappa (C+\hat\rho + \norm{h}_\infty)$, for $x\in\cD$, and the result follows.
\end{proof}

Our next result provides a lower bound on the growth of the solution for large $\rho$.

\begin{lemma}\label{L3.6}
Let Assumption [AP](1)-(2) hold.
For every $\hat\rho>0$ there exists $C_3>0$ such that for every solution $u$ of \eqref{E-AP} with $\rho\geq -\hat\rho$
we have
$$
\rho^+\leq C_3(1+\norm{u^+}_\infty) \leq C_3(1+\norm{u}_\infty).
$$
\end{lemma}

\begin{proof}
Let $\varphi=u-\frac{\rho}{\lambda^*_0}\Phi_1$. Then we have $\varphi\in\cC^{2s+}(\cD)\cap\cC_0(\cD)$. Also,
\begin{eqnarray*}
-L\varphi (x)
&=&
f(x, u)+\rho\Phi_1+h-\rho\Phi_1 \\
&=&
f(x, u)-h\geq f(x, u^+)+f(x, -u^-)-\norm{h}_\infty \geq -C_4(1+ u^+(x)),
\end{eqnarray*}
with a constant $C_4 = C_4(\norm{h}_\infty,\norm{V_2}_\infty, C, \hat\rho)$, where in the last estimate we used Lemma~\ref{L3.5} and \eqref{AP3}.
Thus
$$
L\varphi\leq C_4(1+ u^+)\quad \text{in}\; \cD.
$$
By an application of \cite[Th.~4.12 and Rem~3.9]{BL17b} it then follows that with a constant $C_5$,
$$
\sup_{\cD} (-\varphi)^+\leq C_5 C_4 (1+ \norm{u^+}_\infty)
$$
holds. Pick $x\in\cD$ such that $\Phi_1(x)=1$; this is possible since $\norm{\Phi_1}_\infty=1$ by assumption.
It gives
$$
\frac{\rho}{\lambda^*_0}-u(x) \leq (-\varphi(x))^+\leq C_5 C_4 (1+ \norm{u^+}_\infty),
$$
which, in turn, implies
$$
\rho\leq \lambda^*_0 \left( C_4 C_5  +(1+C_4 C_5)\norm{u^+}_\infty\right),
$$
proving the claim.
\end{proof}

One may notice that we have not used the second condition in \eqref{AP1} so far. The next result makes use of this condition to
establish an upper bound on the growth of $u$.

\begin{lemma}\label{L3.7}
Let Assumption [AP](3) hold. For every $\hat\rho>0$ there exists $C_0$ such that for every solution $u$ of \eqref{E-AP}, for
$\rho\geq -\hat\rho$ we have
\begin{equation}\label{EL3.70}
\norm{u}_\infty\leq C_0.
\end{equation}
In particular, there exists $\rho_2>0$ such that \eqref{E-AP} does not have any solution for $\rho\geq \rho_2$.
\end{lemma}

\begin{proof}
Suppose, to the contrary, that there exists a sequence $(\rho_n, u_n)_{n\in\mathbb N}$ satisfying \eqref{E-AP} with
$\rho_n\geq -\hat\rho$ and $\norm{u_n}_\infty\to \infty$. From Lemma~\ref{L3.5} it follows that $\norm{u^+_n}_\infty
=\norm{u_n}_\infty$. Define $v_n=\frac{u_n}{\norm{u_n}_\infty}$. Then
\begin{equation}\label{EL3.7A}
-L v_n=H_n(x)=\frac{1}{\norm{u_n}_\infty}\left(f(x, u_n) + \rho_n\Phi_1 + h\right)\quad \text{in}\; \cD.
\end{equation}
Since $\norm{H_n}_\infty$ is uniformly bounded by Lemmas~\ref{L3.5}-\ref{L3.6}, it follows by
\cite[Th.~2.6]{B18} that
$$
\sup_{n\geq 1}\; \norm{v_n}_{\cC^\alpha(\Rd)} < \infty.
$$
for some $\alpha>0$.
Hence we can extract a subsequence of $\seq v$, denoted by the original sequence, such that it converges to a
continuous function $v\in\cC_0(\cD)$ in $\cC(\Rd)$. Denote
$$
G_n(x)= \frac{1}{\norm{u_n}_\infty}(f(x, -u_n^-(x))+h(x)-C+ V_2(x)u_n^-(x) - \rho^-_n\Phi_1(x)).
$$
Then using \eqref{AP3} and \eqref{EL3.7A}, we get
$$
-L v_n - V_2 v_n \geq G_n\quad \text{in}\; \cD.
$$
Using Lemma~\ref{L3.1} we have that $v_n$ is a semigroup super-solution, i.e., for every $t>0$
\begin{equation}\label{EL3.7B}
v_n(x)\geq \Exp^x\left[\int_0^{t\wedge\uptau_\cD} e^{\int_0^s V_2(X_p)\D{p}} G_n(X_s)\, \D{s}\right]
+\Exp^x\left[e^{\int_0^t V_2(X_s)\, \D{s}} v_n(X_t)\Ind_{\{t<\uptau_\cD\}}\right].
\end{equation}
Letting $n\to\infty$ in \eqref{EL3.7B} and using the uniform convergence of $G_n$ and $v_n$, we obtain
\begin{equation}\label{EL3.7C}
v(x)\geq \Exp^x\left[e^{\int_0^t V_2(X_s)\, \D{s}} v(X_t)\Ind_{\{t<\uptau_\cD\}}\right]\quad
\text{for all}\; x\in\cD, \; t\geq 0.
\end{equation}
Since $\norm{v}_\infty=1$ and $v\geq 0$ in $\Rd$, it is easily seen from \eqref{EL3.7C} that $v>0$ in $\cD$.
Again, by Lemma~\ref{L3.1}, we see that
$$-L v -V_2 v\geq 0\quad \text{in}\; \cD,$$
in viscosity sense.
Hence it follows that $\lambda^*(\fL-V_2)\geq 0$, contradicting \eqref{AP1}. This
proves the first part of the result. The second part follows by Lemma~\ref{L3.6} and \eqref{EL3.70}.
\end{proof}

The following result will be useful for tackling the super-linear case. We note that the general idea of the
a priori bound below has its origins in the work of Gidas and Spruck \cite{GS81}, see also the more recent
\cite{BDGQ} for a non-local version.

\begin{lemma}\label{L2.8}
Let Assumption [AP](3') hold. Then for every $\hat\rho>0$ there exists $C_0$ such that for every solution $u$
of \eqref{E-AP}, with $\rho\geq -\hat\rho$ we have
\begin{equation}\label{EL2.8A}
\norm{u}_\infty\leq C_0 \max\left\{1, |\rho|^{\frac{1}{p}}\right\}.
\end{equation}
In particular, there exists $\rho_2>0$ such that \eqref{E-AP} does not have any solution for $\rho\geq \rho_2$.
\end{lemma}

\begin{proof}
First we establish \eqref{EL2.8A} for all $\rho\geq 1$. Suppose, to the contrary, that there exists
$(u_n , \rho_n)_{n\in\mathbb N}$, $\rho_n\geq 1$, satisfying \eqref{E-AP} with the property that
\begin{equation}\label{EL2.8B}
\norm{u_n}_\infty \geq n \rho^{\frac{1}{p}}_n, \quad n\geq 1.
\end{equation}
This then implies that $\norm{u_n}_\infty\to \infty$ as $n\to\infty$, and
\begin{equation}\label{EL2.8D}
\norm{u_n}^{-p}_\infty \rho_n\to 0\quad \text{as}\quad n\to\infty.
\end{equation}
Let $x_n\in\cD$ be such that $u(x_n)=u^+(x_n)=\norm{u_n}_\infty$. Such a choice is possible due to Lemma~\ref{L3.5}.
Write
$$
\gamma_n=\norm{u_n}_\infty^{-\frac{p-1}{2s}} \quad \text{and}\; \quad \theta_n=\dist(x_n, \partial\cD).
$$
Using compactness, we may also assume that $x_n\to x_0\in\bar\cD$ as $n\to\infty$. We split the proof into two cases.

\noindent{\bf \emph{Case 1.}} Suppose that $\limsup_{n\to\infty}\frac{\theta_n}{\gamma_n}=+\infty$. Define $w_n(x)=
\frac{1}{\norm{u_n}_\infty} u_n(\gamma_n x+ x_n)$. We then have in $\frac{1}{\gamma_n}(\cD-x_n)$ that
\begin{equation}\label{EL2.8E}
\fL w_n = \frac{1}{\norm{u_n}^p_\infty}\big(f(\gamma_n x+ x_n, u_n(x)) + \rho_n \Phi(\gamma_n x+ x_n) +
h(\gamma_n x+ x_n)\big).
\end{equation}
We choose a subsequence, denoted is the same way, such that $\lim_{n\to\infty}\frac{\theta_n}{\gamma_n}=+\infty$.
Then for any given $k\in\NN$ there is a large enough $n_0$ satisfying $\cB_k(0)\subset \frac{1}{\gamma_n}(\cD-x_n)$
for all $n\geq n_0$. Therefore, the right hand side of \eqref{EL2.8E} is uniformly bounded in $\cB_k(0)$. Since
$\norm{w_n}=w_n(0)=1$, it follows that for some $\alpha>0$, $\norm{w_n}_{\cC^\alpha(\cB_{k/2}(0))}$ is bounded
uniformly in $n$ (see \cite{CS09}). Thus we can extract a subsequence $\seq w$ such that $w_n \to w\in \cC_{\rm b, +}
(\Rd)$ locally uniformly. Hence, by the stability of viscosity solutions
$$
\fL w= a_0(x_0) w^p \quad \text{in}\; \Rd, \quad w(0)=1.
$$
By the strong maximum principle we also have $w>0$. However, no such solution can exist due to the Liouville theorem
\cite[Th.~1.2]{QX14}, and hence we have a contradiction in this case.

\vspace{0.1cm}
\noindent{\bf \emph{Case 2.}} Suppose that $\limsup_{n\to\infty}\frac{\theta_n}{\gamma_n}<+\infty$. First we show that
for a positive constant $\kappa$
\begin{equation}\label{EL2.8C}
\liminf_{n\to\infty}\frac{\theta_n}{\gamma_n}\geq \kappa.
\end{equation}
Note that using Lemma~\ref{L3.5} and Assumption [AP](3') we can find a constant $\kappa_1$ satisfying
$$
\kappa_1(1+\norm{u_n}^{p-1}_\infty)\,\sgn(u_n) u_n \geq f(x, u_n)\quad \text{for}\; x\in\cD, \; n\geq 1.
$$
Indeed, using Assumption [AP](3') it follows that for $u_n(x)\geq \ell$, for some $\ell>0$, we have
$$
f(x, u_n(x))\leq 2 \norm{a_0}_\infty u^p_n(x)\leq 2 \norm{a_0}_\infty \norm{u_n}^{p-1}_\infty u_n(x).
$$
Then the estimate follows from the local Lipschitz property of $f$ and Lemma~\ref{L3.5}. Hence, using \eqref{E-AP}
we obtain
$$
-\fL u_n + \kappa_1(1+\norm{u_n}^{p-1}_\infty)\sgn(u_n) u_n\geq -\rho_n -\norm{h}_\infty\quad \text{in}\; \cD.
$$
Denote by $C_n=\rho_n +\norm{h}_\infty$. Applying Lemma~\ref{L3.1} we get that for $t\geq 0$,
\begin{align*}
\norm{u_n}_\infty =u_n(x_n) &\leq e^{\kappa_1(1+\norm{u_n}^{p-1}_\infty) t} \norm{u_n}_\infty \Prob^{x_n}(\uptau_\cD>t)
+ e^{\kappa_1(1+\norm{u_n}^{p-1}_\infty) t} t C_n.
\end{align*}
It follows from the proof of \cite[Th.~1.1]{B17} that there exist constants $\kappa_2$ and $\eta\in(0, 1)$, not
depending on $x_n$, such that for $t=\kappa_2 \theta_n^{2s}$ we have
$$
\Prob^{x_n}(\uptau_\cD>t)\leq \eta.
$$
Inserting this choice of $t$ in the above expression we obtain
$$
1\leq  e^{\kappa_1(1+\norm{u_n}^{p-1}_\infty) t} \left[\eta + \kappa_2 \theta_n^{2s}\frac{C_n}{\norm{u_n}_\infty}\right]
= e^{\kappa_1(1+\norm{u_n}^{p-1}_\infty) t}
\left[\eta + \kappa_2 \frac{\theta_n^{2s}}{\gamma^{2s}_n}\frac{C_n}{\norm{u_n}^p_\infty}\right].
$$
Thus by the assertion and \eqref{EL2.8D} it follows that for all large $n$ we have
$$
\kappa_1\kappa_2 \theta_n^{2s}(1+\norm{u_n}^{p-1}_\infty) \geq \log \frac{2}{\eta}.
$$
This gives \eqref{EL2.8C}, since $\theta_n\to 0$.

Hence we may assume that, up to a subsequence,
$$
\lim_{n\to\infty}\frac{\theta_n}{\gamma_n}=b\in (0, \infty)
$$
holds. Then using again an argument similar to above, we obtain a positive bounded solution
$$
\fL w= a_0(x_0) w^p \quad \text{in}\; \mathbb{R}^d_+,
$$
see, for instance, the arguments in \cite[Lem.~5.3]{LM17}. This again contradicts \cite[Th.~1.1]{QX14}.

Thus \eqref{EL2.8B} can not hold and this proves our result when $\rho\geq 1$. For the remaining case $\rho\in [-\hat\rho, 1]$,
note that we can rewrite
$$
\rho\Phi_1 + h= \Phi_1 + \tilde{h}\quad \text{where}\quad \tilde{h}=h-\Phi_1 + \rho\Phi_1.
$$
Note that $\norm{\tilde{h}}_\infty$ is uniformly bounded for $\rho\in [-\hat\rho, 1]$.
Then \eqref{EL2.8B} follows from the previous argument. The other claim follows by \eqref{EL2.8A} and Lemma~\ref{L3.6}.
\end{proof}

With the above results in hand, we can now proceed to prove Theorem~\ref{T-AP}. Define
$$
\cA=\big\{\rho\in\RR\; :\; \eqref{E-AP}\; \text{has a viscosity solution}\big\}.
$$
By Lemma~\ref{L3.4} we have that $\cA \neq \emptyset$, and Lemma~\ref{L3.7} and ~\ref{L2.8} imply that $\cA$ is bounded
from above. Define
$$\rho^*=\sup\cA.$$ Note that if $\rho'<\rho^*$, then $\rho'\in\cA$. Indeed, there is $\tilde{\rho}
\in (\rho', \rho^*)\cap\cA$ and the corresponding solution $u^{(\tilde{\rho})}$ of \eqref{E-AP} with $\rho=\tilde{\rho}$
is a super-solution at level $\rho'$, i.e.,
$$
-L u^{(\tilde{\rho})}\geq  f(x, u^{(\tilde{\rho})}) + \rho' \Phi_1 + h(x) \quad \text{in}\; \cD,
\quad \text{and} \quad u=0\quad \text{in}\; \cD^c.
$$
Using Lemma~\ref{L3.3}(3) and from the proof of Lemma~\ref{L3.4} we have a minimal solution of \eqref{E-AP} with $\rho=\rho'$. Next we show
that there are at least two solutions for $\rho< \rho^*$.

Recall that $d:\bar\cD\to [0, \infty)$ is the distance function from the boundary of $\cD$. We can assume that
$d$ is a positive $\cC^1$-function in $\cD$. For a sufficiently small $\varepsilon>0$, to be chosen later, consider
the Banach space
$$
\mathfrak X=\left\{\psi\in \cC_0(\cD)\; :\; \left\|{\frac{\psi}{d^s}}\right\|_{\cC^\varepsilon(\cD)} < \infty\right\}.
$$
In fact, it is sufficient to consider any $\varepsilon$ strictly smaller than the parameter $\alpha < s \wedge (1-s)$
in \cite[Th.~1.2]{RS17}. Since $d^s$ is $s$-H\"{o}lder continuous in $\bar\cD$, it is routine to check that $\mathfrak X
\subset\cC^\varepsilon(\bar\cD)$.

For $\rho\in\RR$ and $m\geq 0$ we define a map $K_\rho:\mathfrak X\to\mathfrak X$ as follows. For $v\in\mathfrak X$,
$K_\rho v =u$ is the unique viscosity solution (see Theorem~\ref{T2.1}(b)) to the Dirichlet problem
$$
-L u + m u = f(x, v) + \rho \Phi_1 + h(x)+ mv \quad \text{in}\; \cD, \quad \text{and} \quad u=0\quad \text{in}\; \cD^c.
$$
It follows from \cite[Th.~1.2]{RS17} that $u\in\mathfrak X$.
\begin{lemma}\label{L3.8}
Let $\rho<\rho^*$. Then there exist $m\geq 0$ and an open $\cO \subset \mathfrak X$, containing the minimal solution,
satisfying $\Deg(I-K_\rho, \cO, 0)=1$.
\end{lemma}

\begin{proof}
We borrow some of the arguments of \cite{DF80} with a suitable modification. Pick $\bar{\rho}\in (\rho, \rho^*)$ and
let $\bar{u}$ be a solution of \eqref{E-AP} with $\rho=\bar{\rho}$. It then follows that
$$
-L \bar{u}  > f(x, \bar{u}) + \rho \Phi_1 + h(x) \quad \text{in}\; \cD \quad \text{and} \quad u=0\quad \text{in}\; \cD^c.
$$
and by Lemma~\ref{L3.3}(i) we have a classical subsolution
$$
-L \underline{u}  < f(x, \underline{u}) + \rho \Phi_1 + h(x) \quad \text{in}\; \cD \quad \text{and} \quad u=0\quad
\text{in}\; \cD^c.
$$
Then Lemma~\ref{L3.3}(3) supplies $\underline{u}\leq\bar{u}$ in $\Rd$, hence the minimal solution $u$ of
\eqref{E-AP} satisfies $\underline{u}\leq u\leq\bar{u}$ in $\Rd$. Note that for every $\psi\in\mathfrak X$, the ratio
$\frac{\psi}{d^s}$ is continuous up to the boundary. Define
$$
\cO=\left\{\psi\in\mathfrak X\; :\; \underline{u}<\psi<\bar{u}\; \text{in}\; \cD, \;\; \frac{\underline{u}}{d^s}<
\frac{\psi}{d^s} < \frac{\bar u}{d^s}\; \text{on}\; \partial\cD, \;\; \norm{\psi}_{\mathfrak X}< r\right\},
$$
where the value of $r$ will be chosen conveniently below. It is evident that $\cO$ is bounded, open and convex. Also,
if we choose $r$ large enough, then the minimal solution $u$ belongs to $\cO$. Indeed, note that
$$
-L(u-\underline{u}) + \left(\frac{f(x, u)-f(x, \underline{u})}{u-\underline{u}}\right)^+(u-\underline{u})\geq 0
\quad \text{in}\; \cD.
$$
By strong maximum principle it follows that $\underline{u}<u$ in $\cD$.
Hence by \cite[Lem.~1.2]{GS16} we have
$$
\min_{\partial\cD}\left(\frac{u}{d^s}-\frac{\underline{u}}{d^s}\right)>0.
$$
The results in \cite{GS16} are proved for $\fL$ (i.e. $k$ constant), but the similar argument works with the barrier function constructed in
\cite[Lem.~3.4]{RS16} giving us Hopf's lemma in our setting.
Similarly, we can compare also $u$ and $\bar{u}$.

We define $m$ to be a Lipschitz constant of $f(x, \cdot)$ in the interval $[\min\underline{u}, \max\bar{u}]$. Also, define
\[
\tilde{f}(x, q)= f \left(x, (\underline{u}(x)\vee q)\wedge\bar{u}(x)\right) + m (\underline{u}(x)\vee q)\wedge\bar{u}(x).
\]
Note that $f$ is bounded and Lipschitz continuous in $q$, and also non-decreasing in $q$. We define another map $\tilde
K_\rho: \mathfrak X \to \mathfrak X$ as follows: for $v\in\mathfrak X$, $\tilde K_\rho v=u$ is the unique viscosity
solution of
\begin{equation}\label{EL3.8A}
-L u + m u = \tilde{f}(x, v) + \rho\Phi + h  \quad \text{in}\; \cD, \quad \text{and} \quad u=0\quad \text{in}\; \cD^c.
\end{equation}
It is easy to check that $K_\rho$ is a compact mapping. Using again \cite[Th.~1.2]{RS16}, we find $r$ satisfying
$$
\sup\left\{\norm{\tilde K_\rho v}_{\mathfrak X}\; :\; v\in\mathfrak X\right\} < r.
$$
We fix this choice of $r$. Using the regularity estimate of \cite{Serra15}, we see that the solution $u$ in \eqref{EL3.8A}
is in $\cC^{2s+}(\cD)$. Therefore,
\begin{align*}
-L(u-\underline{u}) + m(u-\underline{u})
&> \tilde{f}(x, v) -m\underline{u} - f(x, \underline{u})
\geq \tilde{f}(x, \underline{u}) -m\underline{u} - f(x, \underline{u})=0.
\end{align*}
Hence by \cite[Th.~2.1, Lem~1.2]{GS16} we have $\underline{u}<u$ in $\cD$ and
$$
\min_{\partial\cD}\left(\frac{u}{d^s}-\frac{\underline{u}}{d^s}\right)>0.
$$
The other estimates can be obtained similarly. Finally, this implies that $\tilde K_\rho v\in \cO$, for all $v\in\mathfrak X$.
Moreover, $0\notin (I-\tilde K_\rho)(\partial\cD)$. Then by the homotopy invariance property of degree we find that
$\Deg(I-\tilde K_\rho,\cO, 0)=1$. Since $\tilde K_\rho$ coincides with $K_\rho$ in $\cO$, we obtain $\Deg(I-K_\rho, \cO, 0)=1$.
\end{proof}

Similarly as before, define $\cS_\rho:\mathfrak X\to\mathfrak X$ such that for $v\in\mathfrak X$, $u=\cS_\rho v$ is given
by the unique solution of
$$
-L u = f(x, v) + \rho \Phi_1 + h(x) \quad \text{in}\; \cD, \quad \text{and} \quad u=0\quad \text{in}\; \cD^c.
$$
Then the standard homotopy invariance of degree gives that $\Deg(I-\cS_\rho, \cO, 0)=1$. This observation will be helpful
in concluding the proof below.

\begin{proof}[Proof of Theorem~\ref{T-AP}]
Using Lemma~\ref{L3.8} we can now complete the proof by using \cite{DF80,DFS}. Recall the map $\cS_\rho$ defined above, and
fix $\rho<\rho^*$. Denote by $\cO_R$ a ball of radius $R$ in $\mathfrak X$. From Lemmas~\ref{L3.7} and ~\ref{L2.8} we find that
$$
\Deg(I-\cS_{\tilde \rho}, \cO_R, 0)=0\quad \text{for all}\;\, R>0, \; \tilde{\rho}\geq \rho_2.
$$
Using again Lemmas~\ref{L3.7}, ~\ref{L2.8} and \cite[Th.~1.2]{RS17}, we obtain that for every $\hat\rho$ there exists a constant
$R$ such that
$$
\norm{u}_{\mathfrak X} < R
$$
for each solution $u$ of \eqref{E-AP} with $\tilde{\rho}\geq -\hat\rho$. Fixing $\hat\rho>\abs{\rho}$ and the corresponding
choice of $R$, it then follows from homotopy invariance that $\Deg(I-\cS_\rho, \cO_R, 0)=0$. We can choose $R$ large enough so
that $\cO\subset\cO_R$. Since $\Deg(I-\cS_{\rho}, \cO, 0)=1$, as seen above, using the excision property we conclude that
there exists a solution of \eqref{E-AP} in $\cO_R\setminus\cO$. Hence for every $\rho<\rho^*$ there exist at least two
solutions of \eqref{E-AP}. The existence of a solution at $\rho=\rho^*$ follows from the a priori estimates in Lemmas~\ref{L3.7}
and ~\ref{L2.8}, the estimate in \cite[Th.~2.6]{B18}, and the stability property of the viscosity solutions. This completes
the proof of
Theorem~\ref{T-AP}.
\end{proof}

\subsection*{Acknowledgments}
This research of AB was supported in part by an INSPIRE faculty fellowship and a DST-SERB grant EMR/2016/004810.

\end{document}